\documentclass[11pt,a4paper]{article}
\usepackage{amsthm,amssymb,amsfonts,amsmath,lineno,mathtools,graphicx}

\usepackage{epsfig,xspace,algorithmic,algorithm,verbatim}
\usepackage{color}

\newcommand\blfootnote[1]{%
 \begingroup
 \renewcommand\thefootnote{}\footnote{#1}%
 \addtocounter{footnote}{-1}%
 \endgroup
}

\newtheorem{theorem}{Theorem}

\newtheorem{corollary}{Corollary}
\newtheorem{definition}{Definition}

\newtheorem{conjecture}{Conjecture}
\newtheorem{open}{Open problem}

\begin{document}

\title{Antimagic Labelings of Caterpillars}

\author{Antoni Lozano\thanks{Computer Science Department,
Universitat Polit\`ecnica de Catalunya, Spain, {\tt antoni@cs.upc.edu}.} \and
Merc\`e Mora\thanks{Mathematics Department, Universitat Polit\`ecnica de Catalunya, Spain, {\tt
merce.mora@upc.edu}.}
\and Carlos Seara\thanks{Mathematics Department, Universitat Polit\`ecnica de Catalunya, Spain, {\tt
carlos.seara@upc.edu}.}}

\maketitle

\blfootnote{\begin{minipage}[l]{0.3\textwidth} \includegraphics[trim=10cm 6cm 10cm 5cm,clip,scale=0.15]{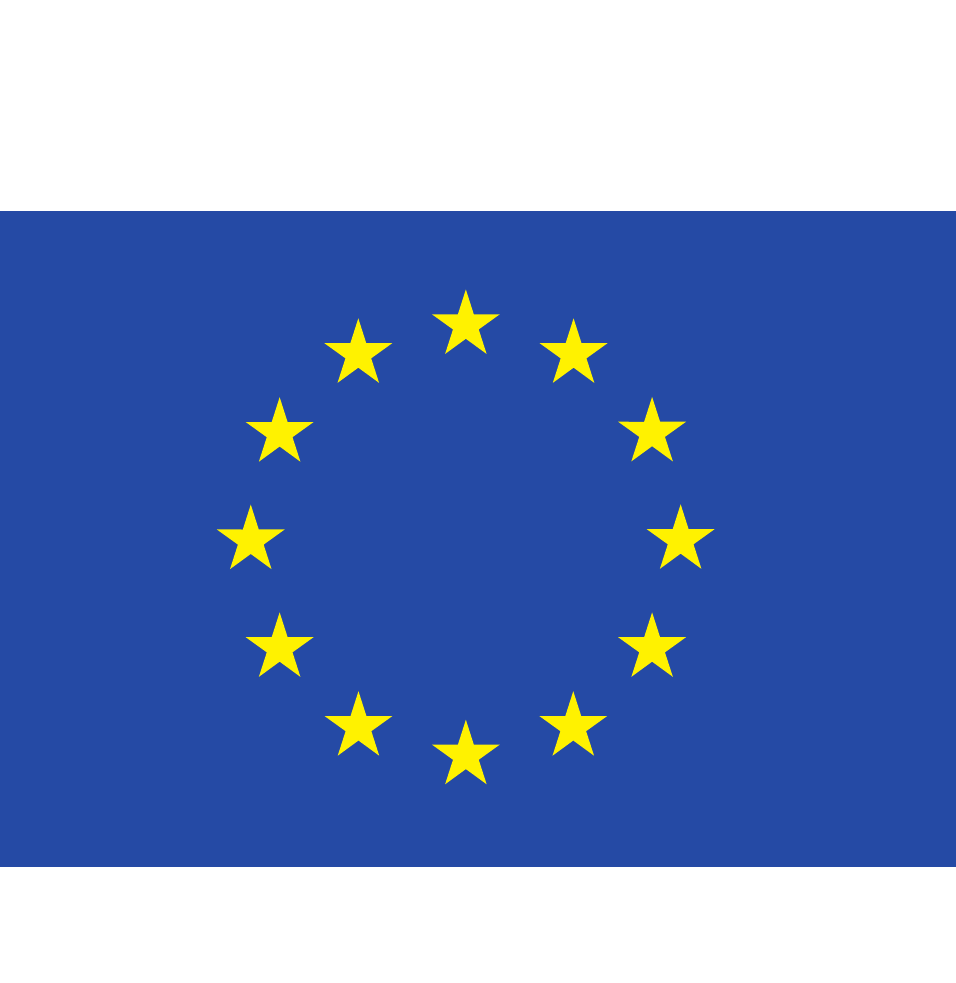} \end{minipage}  \hspace{-2cm} \begin{minipage}[l][1cm]{0.7\textwidth}
      This project has received funding from the European Union's Horizon 2020 research and innovation programme under the Marie Sk\l{}odowska-Curie grant agreement No 734922.
 \end{minipage}}


\begin{abstract}\noindent
A {\em $k$-antimagic labeling} of a graph $G$ is an injection from $E(G)$ to $\{1,2,\dots,|E(G)|+k\}$ such that all vertex sums are pairwise distinct, where the \emph{vertex sum} at vertex $u$ is the sum of the labels assigned to edges incident to $u$. We call a graph \emph{$k$-antimagic} when it has a $k$-antimagic labeling, and \emph{antimagic} when it is $0$-antimagic. Hartsfield and Ringel~\cite{HR} conjectured that every simple connected graph other than $K_2$ is antimagic, but the conjecture is still open even for trees. Here we study $k$-antimagic labelings of \emph{caterpillars}, which are defined as trees the removal of whose leaves produces a path, called its \emph{spine}. As a general result, we use algorithmic aproaches, i.e., \emph{constructive approaches}, to prove that any caterpillar of order $n$ is $(\lfloor (n-1)/2 \rfloor - 2)$-antimagic. Furthermore, if $C$ is a caterpillar with a spine of order $s$, we prove that when $C$ has at least $\lfloor (3s+1)/2 \rfloor$ leaves or $\lfloor (s-1)/2 \rfloor$ consecutive vertices of degree at most $2$ at one end of a longest path, then $C$ is antimagic. As a consequence of a result by Wong and Zhu~\cite{WZ}, we also prove that if $p$ is a prime number, any caterpillar with a spine of order $p$, $p-1$ or $p-2$ is $1$-antimagic.
\end{abstract}

\section{Introduction}\label{sec:int}

All graphs considered in this paper are finite, undirected, connected and simple. Given a graph $G=(V(G),E(G))$ and a vertex $v\in V(G)$, $E_G(v)$ denotes the set of edges incident to $v$ and $d_G(v)=|E_G(v)|$ stands for the degree of $v$ in $G$ (we will just write $E(v)$ and $d(v)$ when $G$ is clear from context).

A \emph{$k$-antimagic labeling} $G$ is an injection $f:E(G)\rightarrow \{1,2,\dots,|E(G)|+k\}$ such that all vertex sums are pairwise distinct, where the \emph{vertex sum} at vertex $v$ is $\sum_{e \in E(v)} f(e)$. A $0$-antimagic labeling is called \emph{antimagic labeling}. A graph is called \emph{$k$-antimagic} if it has a $k$-antimagic labeling, while it is just called \emph{antimagic} if it has an antimagic labeling. The following conjecture from Hartsfield and Ringel~\cite{HR} is well known.

\begin{conjecture}~\cite{HR}\label{conj1}
Every connected graph other than $K_2$ is antimagic.
\end{conjecture}

Classes of graphs which are known to be antimagic include: paths, stars, complete graphs, cycles, wheels, and bipartite graphs $K_{2,m}$, $m\ge 3$ (Hartsfield and Ringel~\cite{HR}); graphs of order $n$ with maximum degree at least $n-3$ (Yilma~\cite{Y}); dense graphs (i.e., graphs having minimum degree $\Omega(\log n)$) and complete partite graphs but $K_2$ (Alon, Kaplan, Lev, Roditty, and Yuster~\cite{AKLRY}, see also~\cite{E}); toroidal grid graphs (Wang~\cite{W});  lattice grids and prisms (Cheng~\cite{Ch07}); $k$-regular bipartite graphs of degree $k\ge 2$ (Cranston~\cite{C}); $k$-regular graphs, $k\ge 3$, with odd degree (Cranston, Liang, and Zhu~\cite{CLZ}); even degree regular graphs (Chang, Liang, Pan, and Zhu~\cite{ChLPZ}); cubic graphs (Liang and Zhu~\cite{LZ}); generalized pyramid graphs (Arumugam, Miller, Phanalasy, and Ryan~\cite{AMPR}); graph products (Wang and Hsiao~\cite{WH}); or Cartesian product of graphs (Cheng~\cite{Ch08} and Zhang and Sun~\cite{ZS}).

Probably the best known result for trees is due to Kaplan, Lev, and Roditty~\cite{KLR}, who proved that any tree having more than two vertices and at most one vertex of degree $2$ is antimagic (see also Liang, Wong, and Zhu~\cite{LWZ}). For antimagic labelings of complete $m$-ary trees see Chawathe and Krishna~\cite{ChK}. Using the Combinatorial Nullstellensatz polynomial method of Alon~\cite{A}, Llad\'o and Miller proved in~\cite{LM} a result about trees that we will point out in Section~\ref{section:general}. However, the conjecture is still open for the general class of trees.

\begin{conjecture}\label{conj2}
Every tree other than $K_2$ is antimagic.
\end{conjecture}

Hefetz~\cite{H} introduced the concept of $(w,k)$-antimagic labeling. Given a graph $G$ and a vertex weight function $w:V(G)\rightarrow \mathbb{N}$, an injection $f: E(G) \rightarrow \{1,2,\dots,|E(G)|+k\}$ is called a \emph{$(w,k)$-antimagic labeling} of $G$ if all vertex sums are pairwise distinct, where the {\em vertex sum} at vertex $v$ is defined, in this context, as $w(v)+\sum_{e \in E(v)} f(e)$. We say that $G$ is \emph{weighted-$k$-antimagic} if for any vertex weight function $w$, $G$ has a $(w,k)$-antimagic labeling. Hefetz~\cite{H} proved that every connected graph of order $n\ge 3$ is weighted-$(2n-4)$-antimagic; Berikkyzy, Brandt, Jahanbekam, Larsen, and Rorabaugh~\cite{BBJLR} improved this result to weighted-$(\lfloor 4n/3 \rfloor)$-antimagic. For more details on antimagic labelings for particular classes of graphs see the dynamic survey from Gallian~\cite{G}.

\begin{definition}
A \emph{caterpillar} $C$ is a tree of order at least 3 the removal of whose leaves produces a path, called the {\em spine} of $C$.
\end{definition}

In this paper we focus on \emph{caterpillars}, which constitute a well-known subclass of trees for which the above Conjecture~\ref{conj2} is still open. A relevant fact is that in our proofs we always use algorithmic approaches, i.e., \emph{constructive approaches}, instead of using the Combinatorial NullStellenSatz method which is the regular technique used in most of the references above. Thus, we obtain a different and more direct way for solving antimagicness problems which can give a clear insight on the difficulties for getting solutions.

\bigskip

\noindent{\bf Our contribution.}
We show in Section~\ref{section:general}, using constructive techniques,  that caterpillars of order $n>2$ are $(\lfloor \frac{n-1}{2} \rfloor-2)$-antimagic. Section~\ref{section:subclasses} is devoted to the study of some classes of caterpillars. Concretely, we prove that if $C$ is a caterpillar with a spine of order $s$ with at least $\lfloor (3s+1)/2 \rfloor$ leaves, then $C$ is antimagic (Subsection~\ref{subsection:leaves});   if $C$  is a caterpillar with a spine of order $s$ and $\lfloor (s-1)/2 \rfloor$ consecutive vertices of degree at most 2 at one end of a longest path, then $C$ is antimagic (Subsection~\ref{subsection:tail}); and  if $p$ is a prime number, any caterpillar with a spine of order $p$, $p-1$ or $p-2$ is $1$-antimagic (Subsection~\ref{subsection:prime}).

\section{Caterpillars}\label{section:general}

We introduce the concept of \emph{leaves excess} for caterpillars.

\begin{definition}
The \emph{leaves excess} of a caterpillar $C$ is $\mathcal{E}(C) =\sum_{d(u)\ge 4} (d(v)-3)$.
\end{definition}

Figure~\ref{fig:caterpillar} shows an example of leaves excess.

\begin{figure}[h]
\centerline{\includegraphics[width=12cm]{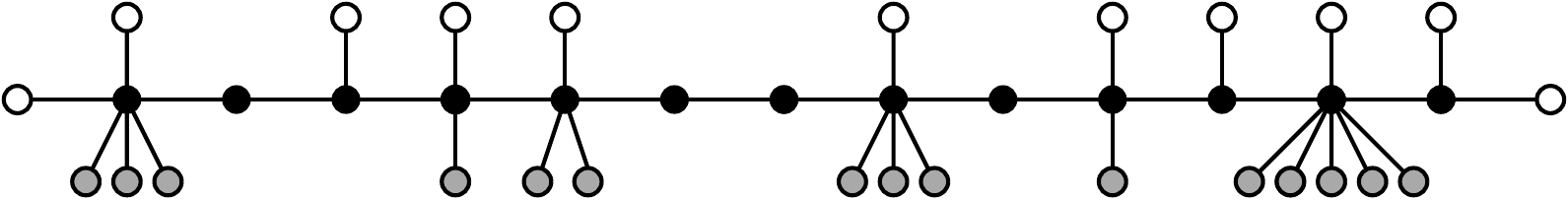}}
\caption{The number of grey leaves equals the leaves excess $\mathcal{E}(C)$.}
\label{fig:caterpillar}
\end{figure}

\begin{theorem}\label{th:caterpillars}
A caterpillar $C$ with a spine of order $s$ is $\max(0,\lfloor \frac{s-1}{2}\rfloor-1-\mathcal{E}(C))$-antimagic.
\end{theorem}

\begin{proof}
Given a caterpillar $C$ with $m$ edges and $\ell$ leaves, let $P$ be a longest path in $C$. It is clear that $P$ has $m-\ell+3$ vertices (which will be called \emph{path vertices}) and $m-\ell+2$ edges (which will be called \emph{path edges}). Also, the number of edges of the spine is $m-\ell=s-1$. In order to define the injection from $E(C)$ to a label set, we choose the values:
\[ i = \Bigl\lceil \frac{m-\ell}{2} \Bigr\rceil + 1,
\;\;\; j = \Bigl\lceil \frac{m+\ell}{2} \Bigr\rceil - 1,
\;\;\; k = \max\left(\mathcal{E}(C),\left\lfloor \frac{m-\ell}{2}\right\rfloor-1\right), \]
and consider a label set $L = \bigcup_{r=1}^4 L_r$, where:

\begin{itemize}
\item $L_1 = \{ 1,\dots,k \}$
\item $L_2 = \{ k+1,\dots,k+i\}$
\item $L_3 = \{ k+i+1,\dots,k+j-\mathcal{E}(C)\}$
\item $L_4 = \{ k+j-\mathcal{E}(C)+1,\dots,k+m-\mathcal{E}(C)\}$
\end{itemize}

Since labels are consecutive across label sets $L_1$, $L_2$, $L_3$, and $L_4$, the total number of labels $|L|$ coincides with the value of the largest label, which equals
\[m + \max\left(\mathcal{E}(C),\left\lfloor \frac{m-\ell}{2} \right\rfloor - 1\right) - \mathcal{E}(C)
= m + \max\left(0,\left\lfloor \frac{m-\ell}{2} \right\rfloor - 1 - \mathcal{E}(C) \right). \]

Therefore, $C$ will be $\max(0,\lfloor \frac{m-\ell}{2} \rfloor-1-\mathcal{E}(C))$-antimagic if there is an injection from $E(C)$ to $L$. Now we show how to construct such an injection.

Labels in $L_2$ and $L_4$ will be assigned to the path edges, while labels in $L_1$ and $L_3$ will be assigned to the non-path edges. We consider two phases to complete the assignment:

\begin{itemize}
\item \emph{Phase 1: Labeling the path edges.}
The assignment of labels to the path $P$ will be done starting with the largest label in $L_2$, then with the largest label in $L_4$, then the previous ones and so on, keeping the alternation of labels until the first ones are reached. We consider two cases depending on whether $s$, the order of the spine, is odd or even. Note that $s=m-\ell+1$, $|L_2|=i$, and $|L_4|=m-j$.

\begin{itemize}
\item If $s$ is odd, then $m-\ell$ is even and from the definition of $i$ and $j$ we have that $i+j=m$. Thus, $|L_2| = i = m-j = |L_4|$ and we can alternate labels along the path $P$ in the following way:
\[ k+i, k+m-\mathcal{E}(C), \dots, k+1, k+j-\mathcal{E}(C)+1.\]

\item If $s$ is even, then $m-\ell$ is odd and $i+j=m+1$. Then, $|L_2| = i = m-j+1 = |L_4|+1$. The alternation of labels in $P$ now ends with the first label in $L_2$:
\[ k+i, k+m-\mathcal{E}(C), \dots, k+2, k+j-\mathcal{E}(C)+1, k+1.\]
\end{itemize}

The previous partial assignment uses all the labels in $L_2$ and $L_4$ to produce partial sums at the inner path vertices ranging between $2k+j+2-\mathcal{E}(C)$ and $2k+i+m-\mathcal{E}(C)$. In addition, the endpoints of path $P$ have the sums $k+i$ and $k+j-\mathcal{E}(C)+1$ (if $s$ is odd), or $k+i$ and $k+1$ (if $s$ is even).

Clearly, all partial sums at the path vertices are different. The vertex sums at the end of the path are smaller than the sums at the inner path vertices, which are obtained summing up two consecutive labels (the smallest sum is $2k+j+2-\mathcal{E}(C)$, which is greater than the possible sums at the extremes of the path: $k+i$, $k+j-\mathcal{E}(C)+1$, and $k+1$). On the other hand, the vertex sums at the inner vertices of $P$ of degree $2$ are all different since, by the way the assignment is defined, sums are strictly decreasing.

\item \emph{Phase 2: Labeling the non-path edges.}
Now, we add the labels in $L_1$ and $L_3$ to the non-path edges in the following way. In the first place, for every path vertex $u$ having degree $d(u) > 3$, we randomly assign labels from $L_1$ to $d(u)-3$ non-path edges incident to $u$. After this step, all the path vertices are incident with at most one non-path edge which has not yet been assigned a label.

Let $E'$ be the list of the still unlabeled non-path edges in non-increasing order of the partial sum of the only path-vertex incident to each of them. Let $L_3'$ be the list of still unused labels of $L_3$ in decreasing order. Assign each label of $L_3'$ to the edge of $E'$ at the same position in the lists. This assignment guarantees all path vertices of degree at least $3$ to have different sums.

Also note that the vertex sums obtained at the non-path vertices (each of them from a unique label in $L_1$ or $L_3$) are smaller than the vertex sums at the path vertices of degree at least $2$, which contain at least one label from $L_4$ and, hence, are greater than any label in $L_1$ or $L_3$.
\end{itemize}

We conclude the proof checking that sums (obtained in Phase 1) at the vertices of degree $2$ never coincide with sums (obtained in Phase 2) at the vertices of degree at least $3$. The largest vertex sum that can be achieved at a vertex of degree $2$ is obtained summing up the largest labels in $L_2$ and $L_4$
\[ (k+i) + (k+m-\mathcal{E}(C)) = 2k+i+m-\mathcal{E}(C) \]
while, on the other hand, the smallest vertex sum that can be achieved at a vertex of degree at least $3$ is obtained summing up the smallest labels in $L_2$, $L_3$, and $L_4$:
\[ (k+1) + (k+i+1) + (k+j-\mathcal{E}(C)+1) = 3k+i+j+3-\mathcal{E}(C). \]
But we have that $2k+i+m-\mathcal{E}(C) < 3k+i+j+3-\mathcal{E}(C)$ if and only if $m-3 < k+j$. By the definition of $j$ and $k$, this last inequality is true if
\[m-3 < \left\lfloor \frac{m-\ell}{2} \right\rfloor-1 +\left\lceil \frac{m+\ell}{2} \right\rceil-1  = m - 2,\]
which is true. Therefore, the sums at both kinds of vertices cannot coincide and the labeling we have constructed is an injection.

Recall that the order of the spine of $C$ is $s= m-\ell+1$. Since according to our labeling, $C$ is $\max(0,\lfloor \frac{m-\ell}{2} \rfloor -1 -\mathcal{E}(C))$-antimagic, we conclude that it is $\max(0,\lfloor \frac{s-1}{2} \rfloor-1-\mathcal{E}(C))$-antimagic and the proof is complete.
\end{proof}

\begin{corollary}
Caterpillars of order $n$  are $(\lfloor \frac{n-1}{2} \rfloor-2)$-antimagic.
\end{corollary}

As a detailed example of Theorem~\ref{th:caterpillars}, consider the caterpillar illustrated in Figure~\ref{fig:caterpillar}. This caterpillar has $39$ vertices, $38$ edges, $26$ leaves, and leaves excess equal to $15$. Thus, $i=\Bigl\lceil \frac{38-26}{2} \Bigr\rceil + 1=7$, $j= \Bigl\lceil \frac{38+26}{2} \Bigr\rceil - 1=31$ and $k=\max\left(15,\left\lfloor \frac{38-26}{2}\right\rfloor-1\right)=15$. The label sets are
$L_1 = \{ 1,\dots,15 \}$;
$L_2 = \{ 16,\dots, 22\}$;
$L_3 = \{ 23,\dots,31\}$;
$L_4 = \{ 32,\dots,38\}$.

\begin{figure}[ht]
\begin{center}
\includegraphics [width=0.95\textwidth]{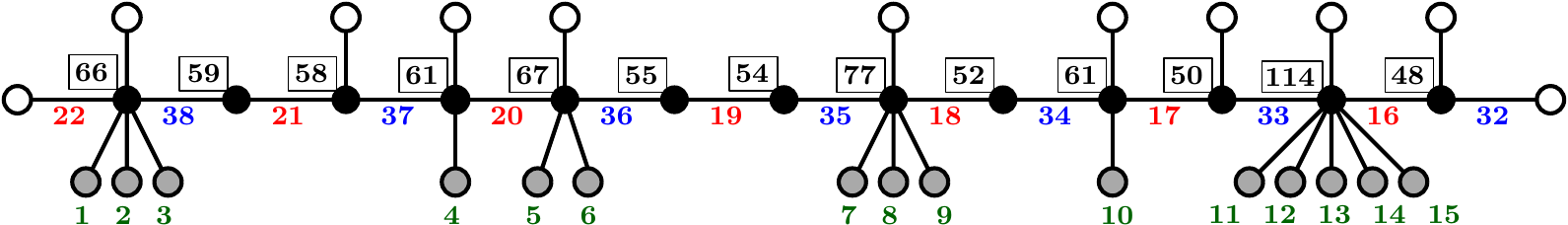}
\vspace{0.6cm}
\newline
\includegraphics [width=1\textwidth]{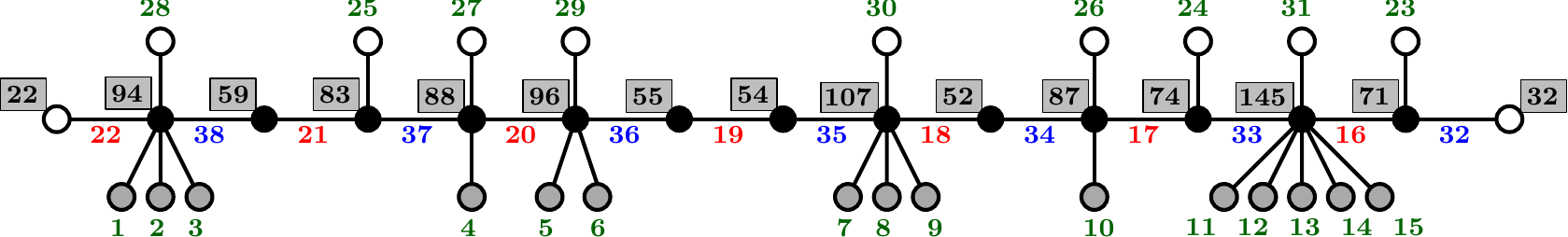}
\caption{Above, red labels belong to $L_2$; blue labels belong to $L_4$ and green labels belong to $L_1$. For each spine vertex, the partial sum is calculated. Below, labels from $L_3$ (also in green) are assigned to the remaining edges (only the sums are shown) and the final sum of each non-leaf vertex is shown (squared).}
\label{fig:exemple}
\end{center}
\end{figure}

First, we assign the labels of the sets  $L_2$ and $L_4$ to the edges of a longest path in the way pointed out in the proof of the theorem. Secondly, we consider a set of  ``leaves excess'' by choosing $d(u)-3$ leaves for each vertex $u$ having degree $d(u)\ge 4$, and we assign randomly the labels of the set $L_1$ to the edges incident to them. Finally, we assign the labels of $L_3$ to the remaining edges in decreasing order of the partial sum of their incident vertices from the longest path (see Figure~\ref{fig:exemple}).

\smallskip

The following result about trees was proved by Llad\'o and Miller~\cite{LM}.

\begin{theorem}(\cite{LM}\label{th:LM}, Thm. 7)
Trees of order $n\ge 3$ having exactly $k$ vertices of degree at least $2$ with no leaf adjacent to them are $k$-antimagic.
\end{theorem}

In fact, Theorem~\ref{th:LM} is stated in terms of the number vertices with no leaf adjacent to them, the \emph{base inner vertices}. Theorem~\ref{th:caterpillars} and Theorem~\ref{th:LM} for caterpillars do not imply each other, as can be seen in the examples shown in Figure~\ref{fig:comparats}.

\begin{figure}[h]\label{fig:comparats}
\centerline{\includegraphics[width=15cm]{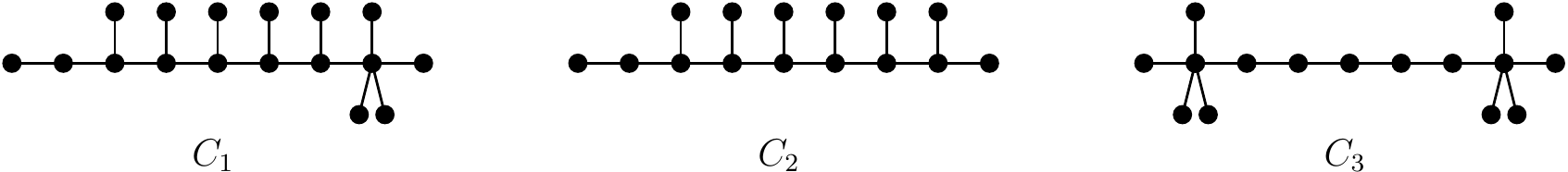}}
\caption{We can derive that $C_1$ is antimagic (from Theorems~\ref{th:caterpillars} and~\ref{th:LM}); $C_2$ is antimagic (from Theorem~\ref{th:LM}), and $2$-antimagic (from Theorem~\ref{th:caterpillars}); $C_3$ is antimagic (from Theorem~\ref{th:caterpillars}) and $5$-antimagic (from Theorem~\ref{th:LM}).}
\end{figure}

\section{Special subclasses}\label{section:subclasses}

\subsection{Caterpillars with many leaves}\label{subsection:leaves}

In the case that the parameter $\mathcal{E}(C)$ is large enough, we can show that caterpillar $C$ is antimagic by direct application of Theorem~\ref{th:caterpillars}.

\begin{corollary}\label{cor:leafy}
Caterpillars with a spine of order $s$ and $\mathcal{E}(C) \ge \left\lfloor \frac{s-1}{2} \right\rfloor - 1$ are antimagic.
\end{corollary}

The following corollary proves the result from Kov\'a\v{r} mentioned in~\cite{K}.

\begin{corollary}
Caterpillars with a spine of order $s$ and at least $\lfloor\frac{3s+1}{2}\rfloor$ leaves are antimagic.
\end{corollary}

\begin{proof}
Let $C$ be a caterpillar with $\ell$ leaves and a spine of order $s$. Let $s_3$ be the number of vertices of the spine with degree at least $3$. From the definition of $\mathcal{E}(C)$, we have that $\mathcal{E}(C)=\ell - s_3 -2$ (see Figure~\ref{fig:caterpillar}). Then, since $s_3\le s$, we obtain
$$\mathcal{E}(C) =\ell - s_3-2\ge \Big\lfloor\frac{3s+1}{2}\Big\rfloor - s -2
= \Big\lfloor\frac{3s+1}{2}-s-1\Big\rfloor-1 = \Big\lfloor\frac{s-1}{2}\Big\rfloor -1.$$
Therefore, the conclusion can be obtained by applying Corollary~\ref{cor:leafy}.
\end{proof}

\subsection{Caterpillars with a long tail}\label{subsection:tail}

The technique used in Theorem~\ref{th:caterpillars} can be adapted to the case when a caterpillar has a ``tail''.

\begin{definition}
A \emph{tail} of order $t$ in a caterpillar $C$ is a path in $C$ containing one endpoint of a longest path in $C$ and $t-1$ more vertices with degree $2$ in $C$.
\end{definition}

\begin{theorem}\label{th:tail}
Caterpillars with a spine of order $s$ and a tail of order $\lfloor \frac{s-1}{2}\rfloor$ are antimagic.
\end{theorem}

\begin{proof}
Let $C$ be a caterpillar with $m$ edges and $\ell$ leaves. Let $P$ be a longest path in $C$. Using the same terminology as in Theorem~\ref{th:caterpillars}, we observe that $C$ has $m-\ell+3$ path vertices and $m-\ell+2$ path edges. Let $T$ be a tail of order $\lfloor \frac{s-1}{2}\rfloor$. Without loss of generality, $T$ can be chosen in such a way that it is a subgraph of $P$ and its endpoint coincides with one of the endpoints of $P$. Recall that $m-\ell=s-1$. In order to define the injection from $E(C)$ to a label set, we choose the values:
\[ i = \Bigl\lceil \frac{m-\ell}{2} \Bigr\rceil + 1, \;\;\;\; j = \Bigl\lceil \frac{m+\ell}{2} \Bigr\rceil - 1, \]
and consider a label set $L = \bigcup_{r=1}^3 L_r$, where:

\begin{itemize}
\item $L_1 = \{ 1,\dots,i \}$
\item $L_2 = \{ i+1,\dots,j\}$
\item $L_3 = \{ j+1,\dots,m\}$
\end{itemize}

Since labels are consecutive across label sets $L_1$, $L_2$, and $L_3$, the total number of labels is $|L| = m$, the value of the largest one. Therefore, $C$ will be antimagic if there is an injection from $E(C)$ to $L$. Now we show how to construct such an injection.

\medskip
Labels in $L_1$ and $L_3$ will be assigned to the path edges, while labels in $L_2$ will be assigned to the non-path edges. We consider two phases to complete the assignment:

\begin{itemize}
\item \emph{Phase 1: Labeling the path edges.}
By assumption, the vertex having degree $1$ in the tail $T$ is one of the endpoints of $P$. We choose the other endpoint in $P$, say $u$, to start our labeling. We assign labels to the path $P$ starting with the largest label in $L_1$, which will be assigned to the edge incident to $u$, then the largest label in $L_3$ will be assigned to the next edge in $P$, then the previous ones and so on, keeping the alternation of labels until the first ones are reached. Later on, we will use the fact that the smallest labels in $L_1$ and $L_3$ have been assigned to the tail $T$. Note that $|L_1|=i$ and $|L_3|=m-j$. Moreover,

\begin{itemize}
\item If $s$ is odd, then $m-\ell$ is even and from the definition of $i$ and $j$ we have that $i+j=m$. Thus, $|L_1| = i = m-j = |L_3|$ and we can alternate labels along the path $P$ in the following way:
\[ i, m, \dots, 1, j+1.\]

\item If $s$ is even, then $m-\ell$ is odd and $i+j=m+1$. Then, $|L_1| = i = m-j+1 = |L_3|+1$. The alternation of labels in $P$ now ends with the first label in $L_1$:
\[ i, m, \dots, 2, j+1, 1.\]
\end{itemize}

The previous partial assignment uses all the labels in $L_1$ and $L_3$ to produce partial sums at the inner path vertices ranging between $j+2$ and $i+m$. In addition, the endpoints of path $P$ have the sums $i$ and $j+1$ (in case $s$ is odd), or $i$ and $1$ (in case $s$ is even).

Clearly, all partial sums at the path vertices are different. The vertex sums at the endpoints of the path are smaller than the sums at the inner path vertices, which are obtained summing up two consecutive labels (the smallest sum is $j+2$, which is greater than the possible sums at the extremes of the path: $i$, $j+1$, and $1$). On the other hand, the vertex sums at the inner vertices of $P$ of degree $2$ in $C$ are all different since, by the way the assignment is defined, sums are strictly decreasing.

\item \emph{Phase 2: Labeling the non-path edges.}
Now, we add the labels in $L_2$ to the non-path edges in the following way. In the first place, for every path vertex $u$ having degree $d(u) > 3$, we randomly assign labels from $L_2$ to $d(u)-3$ non-path edges incident to $u$. After this step, all the path vertices are incident with at most one non-path edge which has not yet been assigned a label.

Let $E'$ be the list of the still unlabeled non-path edges in non-increasing order of the partial sum of the only path-vertex incident to each of them. Let $L_2'$ be the list of still unused labels of $L_2$ in decreasing order. Assign each label of $L_2'$ to the edge of $E'$ at the same position in the lists. This assignment guarantees all path vertices of degree at least $3$ to have different sums.

Also note that the vertex sums obtained at the non-path vertices (each of them from a unique label in $L_2$) are smaller than the vertex sums at the path vertices of degree at least $2$, which contain at least one label from $L_3$ and, hence, they are greater than any label in $L_2$.
\end{itemize}

We conclude the proof checking that sums (obtained in Phase 1) at the vertices of degree $2$ never coincide with sums (obtained in Phase 2) at the vertices of degree at least $3$. On the one hand, the largest vertex sum that can be achieved at a vertex of degree $2$ is $i+m$, since only labels in $L_1$ and $L_3$ are used. On the other hand, the smallest sum that can be obtained at a vertex of degree at least $3$ will be achieved in a vertex belonging to a path vertex but not to the tail. This sum will be composed of two parts:

\begin{enumerate}
\item The smallest label in $L_2$, $i+1$, and

\item The smallest sum obtained in Phase 1 at a vertex not belonging to the tail. The sums obtained at the tail vertices are (from the vertex in $T$ having degree $1$ in $C$) $1,j+2,j+3,j+4,\dots$ or $j+1,j+2,j+3,j+4,\dots$ depending on the parity of the tail order. In any case, the sum $j + \lfloor \frac{m-\ell}{2} \rfloor+ 1$ is the smallest one obtained in Phase 1 for a vertex not belonging to the tail (which has order $\lfloor \frac{m-\ell}{2} \rfloor$).
\end{enumerate}

Therefore, the smallest sum of a path vertex with degree at least $3$ and not belonging to the tail is $(i+1) + (j + \lfloor \frac{m-\ell}{2} \rfloor + 1)$. Then, we need that $i+m < i+j+\lfloor \frac{m-\ell}{2} \rfloor + 2$, which is true if and only if
\[m < \left\lceil \frac{m+\ell}{2} \right\rceil - 1 + \left\lfloor \frac{m-\ell}{2} \right\rfloor + 2 = m + 1, \]
which is true. Therefore, the sums at both kinds of vertices cannot coincide and the labeling we have constructed is an injection.

The order of the spine of $C$ is $s = m-\ell+1$. Since we have assumed that the tail $T$ has order $\lfloor \frac{m-\ell}{2} \rfloor$, this is equivalent to $T$ having order $\lfloor \frac{s-1}{2} \rfloor$, as it is assumed in the statement of the theorem.
\end{proof}

Theorems~\ref{th:caterpillars}, ~\ref{th:LM}, and~\ref{th:tail} for caterpillars do not imply each other, as can be seen in the example shown in Figure~\ref{fig:tail}. An antimagic labeling  for the caterpillar of this example obtained with the procedure described in the proof of Theorem~\ref{th:tail} is also shown.

\begin{figure}[h]
\centerline{\includegraphics[width=8.5cm]{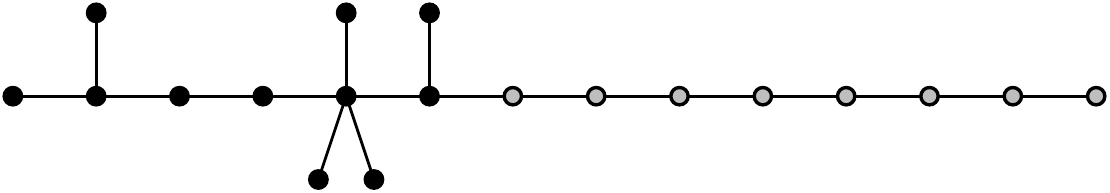}}
\vspace{0.55cm}
\centerline{\includegraphics[width=9.5cm]{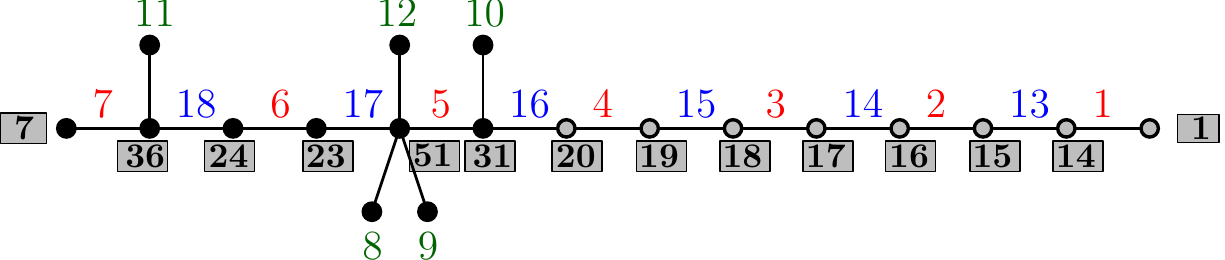}}
\caption{A caterpillar with a tail of order $8$ (gray vertices). Theorem~\ref{th:tail} implies that it is antimagic. Nevertheless, we can only derive that it is $2$-antimagic from Theorem~\ref{th:caterpillars} and $9$-antimagic from Theorem~\ref{th:LM}. An antimagic labeling is given: red labels correspond to $L_1$, blue labels to $L_3$, and green labels to $L_2$. The vertex sums at vertices belonging to a longest path are shown (squared numbers).}
\label{fig:tail}
\end{figure}

\subsection{Caterpillars with a spine of prime order}\label{subsection:prime}

The following theorem is proved by Wong and Zhu~\cite{WZ} using the Combinatorial Nullstellensatz method.

\begin{theorem}(\cite{WZ}\label{th:prime_path}, Thm. 15)
If $n$ is a prime number, then $P_n$ (the path on $n$ vertices) is weighted-$1$-antimagic.
\end{theorem}

By converting a caterpillar into a weighted path, we are able to obtain a $1$-antimagic labeling for caterpillars whose spine order is close to a prime number.

\begin{corollary}
Let $p$ be a prime number. Any caterpillar with a spine of order $p$, $p-1$ or $p-2$ is $1$-antimagic.
\end{corollary}

\begin{proof}
Suppose that $p$ is a prime number. Let $C$ be a caterpillar with $m$ edges, $\ell$ leaves, and a spine $S$ of order $s=p-k$, where $0\le k\le 2$. Let $L=\{1,\dots,m+1\}$ be the label set.

Note that the spine $S$ is part of a longest path with two more vertices than $S$. Then, we define a path $P$ consisting into the spine $S$ of $C$ plus $k$ more vertices from the longest path of $C$ to which $S$ belongs. Thus, $P$ has order $p$, which is a prime number. Now, we assign weights to the vertices of this path $P$.

Let $r=\ell-k$, i.e., $r$ is the number of leaves of $C$ which are not endpoints of $P$. We assign a distinct label in $\{1,\dots,r\}$ to each of the edges incident to these leaves. This labeling produces distinct labels at the $r$ leaves and also a partial sum, called $f(u)$, at each vertex $u$ of the path $P$ (i.e., the sum of the labels of the non-path edges incident to $u$). Define the weight $w(u)$ for each path vertex $u$ of $P$ as follows:

\[ w(u) = \left\{\begin{array}{ll}
                        {f(u)+r}, & \mbox{if $u$ is an endpoint of $P$} \\
                        {f(u)+2r}, & \mbox{otherwise.}
                       \end{array}
                \right.
\]

Now, by Theorem~\ref{th:prime_path}, we get that the path $P$ is weighted-$1$-antimagic. Moreover, let $h:E(P)\rightarrow \{1,\dots,p\}$ be a $1$-antimagic labeling for $P$ according to Theorem~\ref{th:prime_path}. Then, we can argue that
\[ h'(e) = \left\{\begin{array}{ll}
                        {h(e) + r}, & \mbox{if $e \in E(P)$} \\
                        {f(u)}, & \mbox{otherwise, where $u \notin V(P)$ is incident to $e$}
                       \end{array}
                \right.
\]
is a $1$-antimagic labeling for $C$. First, notice that for any $e \in E(P)$,
$$h(e)+r \le p+r = (s+k)+(\ell-k) = s+\ell= (m-\ell+1)+\ell=|L|$$
and, therefore, $h':E(C)\rightarrow L$. Furthermore, all vertex sums are pairwise distinct due to the following facts:

\begin{enumerate}
\item Vertex sums for vertices in $V(C) \setminus V(P)$ have a value at most $r$ and are pairwise distinct since they correspond to the initial labeling of the
leaves of $C$ which are not endpoints of $P$.

\item Vertex sums for vertices in $V(P)$ have a value at least $r+1$ since every path vertex in $V(P)$ is adjacent to some path edge $e$, for which we have defined $h'(e)=h(e)+r>r$. Additionally, the vertex sum at a vertex $u \in V(P)$ using labeling $h'$ coincide with the vertex sum at $u$ using labeling $h$ and weight function $w$. We consider two cases:

\begin{enumerate}
\item $u$ is an endpoint of $P$. Let $e^u$ be the edge in $P$ incident to $u$. The vertex sum at $u$ in $C$ with labeling $h'$ is $h(e^u)+r+f(u)$, while the vertex sum at $u$ in $P$ with weight function $w$ and labeling $h$ is $h(e^u)+w(u)=h(e^u)+f(u)+r$.

\item $u$ is not an endpoint of $P$. Let $e^u_1$ and $e^u_2$ be the edges on the path $P$ incident to $u$. The vertex sum at $u$ in $C$ with labeling $h'$ is
\[(h(e^u_1)+r) + (h(e^u_2)+r) + f(u) = h(e^u_1)+h(e^u_2)+f(u)+2r.\]
On the other hand, the vertex sum at $u$ in $P$ with weight function $w$ and labeling $h$ is $h(e^u_1)+h(e^u_2)+w(u)=h(e^u_1)+h(e^u_2)+f(u)+2r$.
\end{enumerate}

Since $P$ is weighted-$1$-antimagic with the weight function $w$ via labeling $h$, all vertex sums at the vertices in $P$ must be pairwise distinct. Therefore, all vertex sums for vertices in $V(P)$ are also pairwise distinct with labeling $h'$ (and no weights on the vertices).
\end{enumerate}
We conclude that $h'$ is a $1$-antimagic labeling for $C$.
\end{proof}

\section{Open problem}

Although the main problem we leave open in this paper is whether caterpillars are antimagic, an intriguing particular case is the following.

\begin{open}
Are caterpillars with maximum degree $3$ antimagic?
\end{open}

\section{Acknowledgments}
Antoni Lozano is supported by the European Research Council (ERC) under the European Union's Horizon 2020 research and innovation programme (grant agreement ERC-2014-CoG 648276 AUTAR). Merc\`e Mora and Carlos Seara are supported by projects Gen. Cat. DGR 2017SGR1336, 2017SGR1640, MINECO MTM2015-63791-R, and H2020-MSCA-RISE project 734922-CONNECT.


\begin{thebibliography}{XX}

\bibitem{A} N. Alon. Combinatorial Nullstellensatz. \emph{Combinatorics, Probability and Computing}, 8, (1999), 7--29.

\bibitem{AKLRY} N. Alon, G. Kaplan, A. Lev, Y. Roditty, and R. Yuster. Dense graphs are antimagic. \emph{Journal of Graph Theory}, 47(4), (2004), 297--309.

\bibitem{AMPR} S. Arumugam, M. Miller, O. Phanalasy, and J. Ryan. Antimagic labeling of generalized pyramid graphs. \emph{Acta Mathematica Sinica}, 30(2), (2014), 283--290.

\bibitem{BBJLR} Z. Berikkyzy, A. Brandt, S. Jahanbekam, V. Larsen, and D. Rorabaugh. Antimagic Labelings of Weighted and Oriented Graphs, preprint. \emph{arXiv:1510.05070v1} [math.CO] 17 Oct 2015.

\bibitem{ChK} P. D. Chawathe and V. Krishna. Antimagic labelings of complete $m$-ary trees. \emph{Number Theory and Discrete Mathematics}, Trends Math., Birkh\"{a}user, (2002), 77--80.

\bibitem{ChLPZ} F. Chang, Y.-Ch. Liang, Z. Pan, and X. Zhu. Antimagic labeling of regular graphs. \emph{Journal of Graph Theory}, 82(4), (2016), 339--349.

\bibitem{Ch07} Y. Cheng. Lattice grids and prisms are antimagic. \emph{Theoretical Computer Science}, 374(1-3), (2007), 66--73.

\bibitem{Ch08} Y. Cheng. A new class of antimagic Cartesian product graphs. \emph{Discrete Math.}, 308(24), (2008), 6441--6448.

\bibitem{C} D. W. Cranston. Regular bipartite graphs are antimagic. \emph{Journal of Graph Theory}, 60(3), (2009), 173--182.

\bibitem{CLZ} D. W. Cranston, Y.-Ch. Liang, and X. Zhu. Regular graphs of odd degree are antimagic. \emph{Journal of Graph Theory}, 80(1), (2015), 28--33.

\bibitem{E} T. Eccles. Graphs of large linear size are antimagic. \emph{Journal of Graph Theory}, 81(3), (2016), 236--261.

\bibitem{G} J. A. Gallian. A dynamic survey of graph labeling. \emph{The Electronic Journal of Combinatorics}, 5, (2007), \# DS6.

\bibitem{HR} N. Hartsfield and G. Ringel. Pearls in Graph Theory, Academic Press, INC., Boston, 1990 (revised version, 1994), 108--109.

\bibitem{H} D. Hefetz. Antimagic graphs via the Combinatorial NullStellenSatz. \emph{Journal of Graph Theory}, 50(4), (2005), 263--272.

\bibitem{KLR} G. Kaplan, A. Lev, and Y. Roditty. On zero-sum partitions and antimagic trees. \emph{Discrete Math.}, 309, (2009), 2010--2014.

\bibitem{K} P. Kov\'a\v{r}. Antimagic labeling of caterpillars. \emph{9th International Workshop On Graph Labeling}, Open problems, (2016).

\bibitem{LZ} Y.-Ch. Liang and X. Zhu. Antimagic labeling of cubic graphs. \emph{Journal of Graph Theory}, 75 (1), (2014), 31--36.

\bibitem{LWZ} Y.-Ch. Liang, T.-L. Wong, and X. Zhu. Antimagic labeling of trees. \emph{Discrete Math.}, 331, (2014), 9--14.

\bibitem{LM} A. Llad\'o and M. Miller. Approximate results for rainbow labelings. \emph{Periodica Mathematica Hungarica}, 74 (1), (2017), 11--21.

\bibitem{W} T.-M. Wang. Toroidal grids are antimagic. \emph{Proc. 11th Annual International Computing and Combinatorics Conference, COCOON'2005}, LNCS 3595, (2005), 671--679.

\bibitem{WH} T.-M. Wang and C. C. Hsiao. On anti-magic labeling for graph products. \emph{Discrete Math.}, 308(16), (2008), 3624--3633.

\bibitem{WZ} T. Wong and X. Zhu. Antimagic labelling of vertex weighted graphs. \emph{Journal of Graph Theory}, 70(3), (2012), 348--350.

\bibitem{Y} Z. B. Yilma. Antimagic properties of graphs with large maximum degree. \emph{Journal of Graph Theory}, 72(4), (2013), 367--373.

\bibitem{ZS} Y. Zhang and X. Sun. The antimagicness of the Cartesian product of graphs. \emph{Theor. Comput. Sci.}, 410(8-10), (2009), 727--735.

\end{thebibliography}
\end{document}